\documentclass[12pt]{article}
\usepackage[utf8]{inputenc}
\usepackage{amsmath}
\usepackage[english]{babel}

\usepackage[a4paper,top=3cm,bottom=2cm,left=3cm,right=3cm,marginparwidth=1.75cm]{geometry}

\usepackage{amsmath}
\usepackage{graphicx}
\usepackage[colorinlistoftodos]{todonotes}
\usepackage[colorlinks=true, allcolors=blue]{hyperref}
\usepackage{amssymb}
\usepackage{amssymb}
\usepackage{amsthm}
\usepackage{authblk}
\usepackage{amsmath}
\usepackage{amsfonts}
\usepackage{amssymb}
\usepackage{amscd}
\usepackage{latexsym}
\usepackage{breqn}

\theoremstyle{plain}
\newtheorem{thm}{Theorem}[section]
\newtheorem{lem}[thm]{Lemma}
\newtheorem{prop}[thm]{Proposition}
\newtheorem{cor}[thm]{Corollary}
\newtheorem{defn}[thm]{Definition}

\theoremstyle{definition}

\newtheorem{exmp}{Example}[section]

\theoremstyle{remark}
\newtheorem*{rem}{Remark}

\title{Nearly invariant subspaces with applications to truncated Toeplitz operators }
\author{Ryan O'Loughlin \\ E-mail address: \href{mailto:mm1rol@leeds.ac.uk}{mm12rol@leeds.ac.uk}}
\affil{School of Mathematics, University of Leeds, Leeds, LS2 9JT, U.K.}
\date{}
\begin{document}
\maketitle
\begin{abstract}
In this paper we first study the structure of the scalar and vector-valued nearly
invariant subspaces with a finite defect. We then subsequently produce some
fruitful applications of our new results. We produce a decomposition theorem for the vector-valued nearly
invariant subspaces with a finite defect. More specifically, we show every vector-valued nearly
invariant subspace with a finite defect can be written as the isometric image
of a backwards shift invariant subspace. We also show that there is a link
between the vector-valued nearly invariant subspaces and the scalar-valued
nearly invariant subspaces with a finite defect. This is a powerful result which allows us to gain insight in to the structure of scalar subspaces of the Hardy space using vector-valued Hardy space techniques. These results have far reaching applications, in particular they allow us to develop an all encompassing approach to the study of the kernels of: the Toeplitz operator, the truncated Toeplitz operator, the truncated Toeplitz operator on the multiband space and the dual truncated Toeplitz operator.

\end{abstract}

\section{Introduction}
The purpose of this paper is to study vector and scalar-valued nearly $S^*$-invariant subspaces of the Hardy space defined on the unit disc.  We first produce some results on the structure of nearly $S^*$-invariant subspaces with a finite defect, in particular we produce a powerful tool which allows us to relate the vector-valued nearly $S^*$-invariant subspaces to scalar-valued nearly $S^*$-invariant subspaces with a finite defect. These results then allow us to adopt a previously unknown universal approach to the study of the kernel of: the Toeplitz operator, the truncated Toeplitz operator, the dual truncated Toeplitz operator and the truncated Toeplitz operator on the multiband space (all to be defined later).

We denote $\mathbb{T}$ to be the unit circle and $\mathbb{D}$ to be the open unit disc. The vector-valued Hardy space is denoted $H^2(\mathbb{D}, \mathbb{C}^n)$ and is the Hilbert space defined to be a column vector of length $n$ with each coordinate taking values in $H^2$; background theory on the classical Hardy space $H^2$ can be found in \cite{nikolski2002operators, duren1970theory}. The backwards shift on the space $H^2(\mathbb{D}, \mathbb{C}^n)$ is defined by 
$$
S^* \begin{pmatrix}
f_1 \\
\vdots \\
f_n
\end{pmatrix} (z) =  \frac{\begin{pmatrix}
f_1 (z) \\
\vdots \\
f_n (z)
\end{pmatrix}  - \begin{pmatrix}
f_1 (0) \\
\vdots \\
f_n (0)
\end{pmatrix}}{z}.
$$
If we denote $\overline{H^2_0} = \{ \overline{f} : f \in H^2, f(0) = 0 \}$, then it is readily checked that $\overline{H^2_0}$ is the orthogonal complement of $H^2$ in $L^2 (\mathbb{T})$. Then in the scalar case (i.e. when $n=1$) using Beurling's Theorem one can then deduce that all non-trivial $S^*$-invariant subspaces are of the form $K_{\theta} = \theta \overline{H^2_0} \cap H^2$ for some inner function $\theta$. We call $K_{\theta}$ a model space and further information on model spaces can be found in \cite{cima2000backward}. One can further check that for distinct $\lambda_i \in \mathbb{D}$ if $\theta = \prod_{i} \frac{z- \lambda_i}{1 - \overline{\lambda_i}z}  $, then $K_{\theta}$ is the span of Cauchy kernels $k_{{\lambda_i}}(z) = \sum_{n=0}^{\infty} (\overline{\lambda_i} z)^n$. The Cauchy kernel $k_{\lambda_i}$ is the eigenvector of the backwards shift with eigenvalue $\overline{\lambda_i}$.
\begin{defn}
A closed subspace $M \subseteq H^2 ( \mathbb{D} , \mathbb{C}^n )$ is said to be nearly $S^*$-invariant with defect $m$ if and only if there exists a $m$-dimensional subspace $D$ (which may be taken to be orthogonal to $M$) such that if $f \in M$ and $f(0) $ is the zero vector then $S^* f \in M \oplus D$.
\vskip 0.1cm
\noindent If $M$ is nearly $S^*$-invariant with defect 0 then it is said to be nearly $S^*$-invariant.
\end{defn}
Using orthogonal decomposition we can write $L^2 = \overline{H^2_0} \oplus K_{\theta} \oplus \theta H^2$. We define $P_{\theta}: L^2 \to K_{\theta}$ to be the orthogonal projection.

\begin{defn}The truncated Toeplitz operator $A_g^{\theta}: K_{\theta} \to K_{\theta}$ having symbol $g \in L^2$ is the densely defined operator 
$$
A_g^{\theta} (f) = P_{\theta} (g f)
$$
having domain 
$$
\{ f \in K_{\theta} : g f \in L^2 \}.
$$
\end{defn}

The concept of (scalar) nearly backward shift invariant subspaces was first introduced by Hitt in \cite{hitt1988invariant} as a generalisation to Hayashi’s results concerning Toeplitz kernels in \cite{MR853630}. These spaces were then studied further by Sarason \cite{sarason1988nearly}. The study of nearly backwards shift invariant subspaces was then generalised to the vectorial case in \cite{MR2651921}, and generalised to include a finite defect in \cite{chalendar2019beurling}. Kernels of Toeplitz operators are the prototypical example of nearly $S^*$- invariant subspaces. 

Truncated Toeplitz operators were introduced in \cite{sarason2007algebraic}, and over the past decade there have been many further publications studying their properties. The applications of truncated Toeplitz operators are vast, ranging from purely mathematical to more applied problems. From a purely mathematical perspective one can use the Sz.-Nagy-Foiaş model theory for Hilbert space contractions (see \cite{nikolski2002operators}) to show that every Hilbert space contraction $T$ having defect indices $(1,1)$ and such that $\lim_{n \to \infty} (T^*)^n $ (SOT) is unitarily equivalent to $A_z^{\theta}$, for some inner function $\theta$. This can be generalised to produce similar results for arbitrary defect indices. Another notable application of truncated Toeplitz operators within pure mathematics comes from the Carathéodory and Pick problems \cite{pickproblem}, where truncated Toeplitz operators with an analytic symbol appear naturally. From a more applied perspective truncated Toeplitz operators have links to control theory and electrical engineering. More specifically when one is considering an extremal problem posed over $H^{\infty}$, the solution of such a problem can be solved by computing the norm of a Hankel operator, and the norm of the Hankel operator can in turn be shown to equal the norm of an analytic truncated Toeplitz operator. This is shown explicitly as equation 2.9 in \cite{hankel}.

Although truncated Toeplitz operators share many properties with the classical Toeplitz operator it is easily checked that the kernel of a truncated Toeplitz operator is not nearly $S^*$-invariant.  This motivates our study for section 2 where we show under certain conditions the kernel of a truncated Toeplitz operator is in fact nearly $S^*$-invariant with defect 1. In many cases the study of Toeplitz operators becomes greatly simplified when the operator has an invertible symbol; in section 2 we also show that the symbol of a truncated Toeplitz operator may be chosen to be invertible in $L^{\infty}$.

In section 3 we prove a powerful result that shows for any $i \in \{ 1 \hdots n \}$ the first $i$ coordinates of a vector-valued nearly $S^*$-invariant subspace is a nearly $S^*$-invariant subspace with a finite defect. We then generalise Theorem 3.2 in \cite{MR2651921} and Corollary 4.5 in \cite{chalendar2019beurling} to find a Hitt-style decomposition for the vector-valued nearly $S^*$-invariant subspaces with a finite defect.

In section 4 we show that in all cases the kernel of a truncated Toeplitz operator is a nearly $S^*$-invariant subspace with defect 1; this then allows us to decompose the kernel in to an isometric image of a model space. The approach of decomposing a kernel in to an isometric image of a model space much resembles the works of Hayashi \cite{MR853630} and Hitt \cite{hitt1988invariant} for the  classical Toeplitz operator. We also make the observation that we can decompose the kernel of a truncated Toeplitz operator in to a nearly $S^*$-invariant subspace multiplied by a power of $z$ (where $z \in \mathbb{D}$ is the independent variable). Then using the results of \cite{hitt1988invariant}, this observation also gives us a second method to decompose the kernel in to a isometric image of a model space. Furthermore we show that in general our two choices of decomposition of the kernel of a truncated Toeplitz operator yield different results. 

In section 5 we study the kernel of dual truncated Toeplitz operator. Dual truncated Toeplitz operators have been studied in both \cite{dual1,dual2} as well as many other sources. The kernel of a dual truncated Toeplitz operator has been studied in \cite{camara2019invertibility}. Although the domain of the dual truncated Toeplitz operator is not a subspace of $H^2$ we still can use similar recursive techniques used in previous sections to decompose the the kernel in to a fixed function multiplied by a $S^*$-invariant subspace.

In section 6 we study the truncated Toeplitz operator on the multiband space. We show every truncated Toeplitz operator on a multiband space is unitarily equivalent to an operator which has kernel nearly $S^*$-invariant with defect 2. This allows us to apply our previously developed theory to give a decomposition for the kernel of the truncated Toeplitz operator on a multiband space in terms of $S^*$-invariant subspaces.

\subsection{Notations and convention}
\begin{itemize}
    \item From section 3 onward we assume the symbol of any Toeplitz operator (denoted $g$) is bounded and hence the Toeplitz operator is bounded.
    \item Throughout we let $\theta$ be an arbitrary inner function.
    \item We use the notation $f^i / f^o$ to denote the inner/outer factor of $f$.
    \item $GCD$ stands for greatest common divisor, and the greatest common divisor of two inner functions is always taken to be an inner function.
    \item All limits are taken in the $H^2(\mathbb{D}, \mathbb{C}^n)$ sense unless otherwise stated.
    \item All subspaces of $H^2(\mathbb{D}, \mathbb{C}^n )$ are assumed closed unless otherwise stated.
\end{itemize}

\section{Preliminary results}

\begin{thm}\label{suggestinvariance}
For any $g \in L^2$ we write $g = g^- + g^+$ where $g^- \in \overline{H^2_0}$ and $g^+ \in H^2$. If $\overline {g^-}$ is not cyclic for the backwards shift then there exists a $\tilde{g} \in L^2$ such that $A^{\theta}_g = A^{\theta}_{\tilde{g}}$ and $\tilde{g}^{-1} \in H^{\infty}$.
\end{thm}

\begin{proof}
Theorem 3.1 of \cite{sarason2007algebraic} shows that $A_{g_1}^{\theta} = A_{g_2}^{\theta}$ if and only if $g_1 - g_2 \in \overline{\theta H^2} + \theta H^2$, so we may initially assume without loss of generality that $g \in \overline{ K_{\theta}} \oplus K_{\theta}$. Using Lemma 2.1 in \cite{o2020toeplitz} we can construct an outer function $u $ such that $|u| = 2 |g| + 1$, furthermore $u \in L^2$ so $u \in H^2$. Then it follows that for any inner function $\alpha$
\begin{equation}\label{alphatothoeta}
     g - \overline{ \alpha u} 
\end{equation}
has the property that 
$$
 | g - \overline{ \alpha u} | \geqslant |u| - |g| > |g| + 1 > 0
$$
almost everywhere on $\mathbb{T}$, and so $(g - \overline{ \alpha u })^{-1} \in L^{\infty}$. Our construction of $u$ shows $|\frac{1}{u}| \leqslant 1 $ and as the reciprocal of an outer function in is outer, we have $\frac{1}{u}$ is outer and in $L^{\infty}$, so $\frac{1}{u} \in H^{\infty}$. Furthermore by Corollary 4.9 in \cite{o2020toeplitz} we can say $\frac{1}{u} \in H^2$ is non-cyclic for $S^*$ and hence must lie in a model space $K_{\Phi}$. Define $\tilde{g} := ( g - \overline{ \Phi \theta u} )$, then as previously stated $\tilde{g}^{-1} \in L^{\infty}$. We now show $\tilde{g}^{-1} = \sum_{k=0}^{\infty} (-1) g^k (  \Phi \theta \frac{1}{\overline{u}})^{k+1}$ where the limit is taken in the sense of uniform convergence. We write $\tilde{g}^{-1}_{N}$ to be $\sum_{k=0}^{N} (-1) g^k ( \Phi \theta \frac{1}{\overline{u}})^{k+1}$ then we have $||\tilde{g}^{-1}_{N} - \tilde{g}^{-1}||_{\infty}$ is equal to $$
|| \tilde{g}^{-1} \tilde{g} ( \tilde{g}^{-1}_{N} - \tilde{g}^{-1} )||_{\infty} \leqslant ||\tilde{g}^{-1}||_{\infty} || \tilde{g}^{-1} \tilde{g}^{-1}_{N} - 1 ||_{\infty} \leqslant ||\tilde{g}^{-1}||_{\infty} ||g^N ( \Phi \theta \frac{1}{\overline{u}})^{N} ||_{\infty}.
$$
By our construction of $u$ this is less than $||\tilde{g}^{-1}||_{\infty} (\frac{1}{2})^{N}$, which clearly converges to 0. Now our choice of $\Phi $ ensures that $ \Phi \frac{1}{\overline{u}} \in H^{\infty}$, we also have $ \theta g \in H^2$. This means $(-1) g^k (  \Phi \theta \frac{1}{\overline{u}})^{k+1} \in H^2 $ and is bounded by 1 so must actually lie in $H^{\infty}$, so $\tilde{g}^{-1}$ (being the uniform limit of a sequence in $H^{\infty}$) must also be in $H^{\infty}$.
\end{proof} 

\noindent Examining the first part of the above proof we can also deduce the following proposition.

\begin{prop}
For any $g \in L^2$ there exists a $\tilde{g} \in L^2$ such that $A_{g}^{\theta} = A_{\tilde{g}}^{\theta}$ and $\tilde{g}^{-1} \in L^{\infty}$.
\end{prop}
\begin{proof}
In \eqref{alphatothoeta} if we set $\alpha$ to equal $\theta$, keep our construction of $u$ the same and define $\tilde{g} =  g - \overline{\alpha u } $ then $A_{g}^{\theta} = A_{\tilde{g}}^{\theta}$. Furthermore the computation immediately after \eqref{alphatothoeta} shows $\tilde{g}^{-1} \in L^{\infty}$.
\end{proof}
This has an interesting relation to Sarason's question posed in \cite{sarason2007algebraic}; which is whether every bounded truncated Toeplitz operator has a bounded symbol. This was first shown to have a negative answer as Theorem 5.3 in \cite{baranov2010bounded}, and further results in \cite{baranov2010symbols} characterise the inner functions $\theta$ which have the property that every bounded truncated Toeplitz operator on $K_{\theta}$ has a bounded symbol.

These results suggest that under certain circumstances $\ker A_g^{\theta}$ may be a nearly invariant subspace with a finite defect. This is because $f \in \ker A_g^{\theta}$ if and only if $f \in K_{\theta}$ and 
$$
gf \in \overline{H^2_0} \oplus \theta H^2,
$$
so if $f(0) = 0$ and $f \in \ker A_g^{\theta}$ then we must have 
$$
\frac{gf}{z} \in \overline{H^2_0} + { \rm span} \{ S^* ( \theta ) \} + \theta H^2.
$$
This may lead us to believe that $\ker A_g^{\theta}$ is a nearly $S^*$-invariant subspace with a defect given by $g^{-1} { \rm span} \{ S^* ( \theta ) \}$, but the issue here is $g^{-1} S^* (\theta)$ need not necessarily lie in $K_{\theta}$ or even $H^2$. Theorem \ref{suggestinvariance} shows us that under some weak restrictions we can choose our non-unique symbol $g$ so that $g^{-1} S^* ( \theta ) \in H^2$, but to fully understand $\ker A_g^{\theta}$ as a nearly invariant subspace with a defect we must study vector-valued nearly invariant subspaces with a defect.
\section{Vector-valued nearly invariant subspaces with a defect}

Let $M \subseteq H^2(\mathbb{D},\mathbb{C}^n)$ be a nearly invariant subspace for the backwards shift with a finite defect space $D$ and let $\dim D = m$. If not all functions in $M$ vanish at $0$ then we define $ W := M \ominus (M \cap zH^2( \mathbb{D},\mathbb{C}^n)) $ and Corollary 4.3 in \cite{MR2651921} shows that $r := \dim W \leqslant n$, in this case we let $W_1 \hdots W_r$ be an orthonormal basis of $W$. For $i = 1 \hdots n $ we let $P_{i}: H^2( \mathbb{D}, \mathbb{C}^n) \to H^2( \mathbb{D}, \mathbb{C}^i )$ be the projection on to the first $i$ coordinates. 

\begin{thm}\label{main}
For any $i \in \{ 1 \hdots n \}$, $M_i := P_{i}( M)$ is a (not necessarily closed) nearly invariant subspace with a defect space $ \left( \frac{ { \rm span} \{ { P_{i}(W_1), \hdots P_{i}(W_r)} \} }{z} \cap H^2(\mathbb{D}, \mathbb{C}^i) \right) + P_i (D) $.
\end{thm}

\begin{proof}

We first consider the case when not all functions in $M$ vanish at 0. Let $f_i \in M_i$, then $f_i$ is the first $i$ entries of some $F \in M$. We write $F$ as
 $$
 F = a_1 W_1 + \hdots a_r W_r + F_1,
 $$
 where $a_1 \hdots a_r \in \mathbb{C}$ and $F_1 \in M \cap z H^2(\mathbb{D},\mathbb{C}^n)$. So if $f_i(0) $ is the zero vector, we then have $f_i(0) $ is zero and $F_1 (0)$ is zero, which forces $P_i ( a_1 W_1 + \hdots a_r W_r) $ to be zero. So 
 $$
 \frac{f_i}{z} - \frac{P_i ( a_1 W_1 + \hdots a_r W_r)}{z} = P_i \left( \frac{ F_1}{z} \right) \in M_i + P_{i}(D),
 $$
 which means
 $$
 \frac{f_i}{z} \in M_i +  \left( \frac{ { \rm span} \{ { P_{i}(W_1), \hdots P_{i}(W_r)} \}}{z} \cap H^2 \right) + P_i (D).
 $$
 
In the case when all functions in $M$ vanish at $0$ then $W = \{ 0 \}$ and we would just have $\frac{F}{z} \in M + D$, so $\frac{f_i}{z} \in M_i + P_i (D)$.
\end{proof}

\begin{rem}
If $W = \{ 0 \}$ we can interpret $ \left( \frac{ { \rm span} \{ { P_{i}(W_1), \hdots P_{i}(W_r)} \} }{z} \cap H^2(\mathbb{D}, \mathbb{C}^2) \right)$ to be the zero vector.
\end{rem}

\begin{cor}\label{main2}
With the same assumptions as in Theorem \ref{main}, if $m = 0$ i.e. if $M$ is a nearly $S^*$-invariant subspace, then $M_i$ is a (not necessarily closed) nearly $S^*$-invariant subspace with a defect space $ \left( \frac{ { \rm span} \{ { P_{i}(W_1), \hdots P_{i}(W_r)} \} }{z} \cap H^2(\mathbb{D}, \mathbb{C}^i) \right) $.
\end{cor}

To further build on this result we will now give a Hitt style decomposition for a vector-valued nearly invariant subspace with a finite defect. This style of decomposition was first introduced by Hitt in \cite{hitt1988invariant} when he decomposed the nearly $S^*$-invariant subspaces. This was then generalised to the vectorial case as Corollary 4.5 in \cite{MR2651921}. This style of proof was then adapted to produce a similar result for the (scalar) nearly invariant subspace with a defect, which is Theorem 3.2 in \cite{chalendar2019beurling}.

\vskip 0.5cm

For a Hilbert space $\mathcal{H}$ and $x, y \in \mathcal{H}$ we define $x\otimes y (f) = \langle f, y \rangle x$. We say an operator $T$ on $\mathcal{H}$ belongs to the class  $C_{.0}$ if for all $x \in \mathcal{H}, \lim_{n \to \infty} ||(T^*)^n x || = 0.$
Consider a subspace $M$ which is nearly $S^*$-invariant with defect 1, so that $D= { \rm span} \{ e_1 \}$, say, where $||e_1 || =1.$

Suppose first that not all functions in $M$ vanish at $0$, then $1 \leqslant r = \dim W \leqslant n$. Let $F_0$ be the matrix with columns $W_1 \hdots W_r$, and let $P_{W}$ be the orthogonal projection on to $W$. For each $F \in M$ we may write 
$$
F = P_{W}(F) + F_1 = F_0 \begin{pmatrix}
a_0^1 \\
\vdots \\
a_0^r \\
\end{pmatrix} + F_1 .
$$
Now as $F_1 (0) = 0$ we have $S^* (F_1) = G_1 + \beta_1 e_1$, where $G_1 \in M$ and $\beta_1 \in \mathbb{C}$. Thus
$$
F(z) = F_0(z) A_0 + z G_1(z) + z \beta e_1(z) ,
$$
where $A_0 = \begin{pmatrix}
a_0^1 \\
\vdots \\
a_0^r \\
\end{pmatrix}$. Moreover since the family $\{ W_i \}_{i=1 \hdots r}$ forms an orthonormal basis of $W$, we obtain the following identity of norms:
$$
||F||^2 = ||F_0 A_0 ||^2 + || F_1 ||^2 = ||A_0||^2 + ||G_1||^2 + |\beta_1|^2.
$$
We may now repeat this process on $G_1$ to obtain $G_1 = P_{W}(G_1) + F_2$, and $S^* (F_2) = G_2 + \beta_2 e_1$, so $G_1 = F_0 A_1 +  zG_2 + z \beta_2 e_1$. We iterate this process to obtain
\begin{equation}\label{star}
F(z) = F_0(z) ( A_0 + A_1 z + \hdots A_{n-1} z^{n-1} ) + z G_n (z) + \left(\beta_1 z + \hdots + \beta_n z^n \right)e_1(z) ,
\end{equation}
where
$$
||F||^2 = \sum_{k=0}^{n-1}||A_k||^2 + ||G_n||^2 + \sum_{k=1}^{n} |\beta_k|^2 .
$$

We now argue $||G_n|| \to 0$ as $n \to \infty$. We can write $G_n = P_1 S^* P_2 (G_{n-1})$, where $P_1$ is the projection with kernel $\langle e_1 \rangle$ and $P_2$ is the projection with kernel $ { \rm span} \{ W_1 \hdots W_r \}$. For all $n \geqslant 1$ we may write $G_{n+1} = P_1 R^{n-1} ( S^* P_2 ( G_1 ) )$, where $R= S^* P_2 P_1$ and so
\begin{equation}\label{limit}
||G_{n+1} || \leqslant ||P_1|| ||R^{n-1} (S^* P_2 (G_1))||.
\end{equation}
As $e_1$ is orthogonal to $W$ we have 
$$P_2 P_1 = P_1 P_2 = Id - e_1 \otimes e_1 - \sum_{j=1}^{r} W_j \otimes W_j,$$
and so the adjoint of $R$ is
$$
P_1 P_2 S = S - e_1 \otimes S^*(e_1) - \sum_{j=1}^{r} W_j \otimes S^*(W_j).
$$
We now apply the second assertion of Proposition 2.1 from \cite{MR2651921} to show the adjoint of $R$ is of class $C_{.0}$, and so $R^{n-1}$ applied to $S^* P_2 (G_1)$ converges to $0$; now from \eqref{limit} we see $||G_{n+1}|| \to 0$. As a consequence taking limits in \eqref{star} we may write 
$$
F(z) = \lim_{n \to \infty} \left( F_0(z) ( A_0 + A_1 z + \hdots A_{n-1} z^{n-1} )  + (\beta_1 z + \hdots + \beta_n z^n )e_1(z) \right).
$$
We denote $a_n(z) =  F_0(z) \left( A_0 + A_1 z + \hdots A_{n-1} z^{n-1} \right)$, and $a_0(z) =  F_0 \left( \sum_{k=0}^{\infty} A_k z^k \right)$, where $\left( \sum_{k=0}^{\infty} A_k z^k \right)$ is taken in the $H^2(\mathbb{D}, \mathbb{C}^n)$ sense  (this is defined by the equality of norms given immediately after \eqref{star}). Then in the $H^1 ( \mathbb{D}, \mathbb{C}^n )$ norm we must have 
$$
|| a_n(z)- a_0(z) || = || F_0 \sum_{k=n}^{\infty} A_k z^k || \leqslant ||W_1 \sum_{k=n}^{\infty} a_k^1 z^k || + \hdots + ||W_r \sum_{k=n}^{\infty} a_k^r z^k ||.
$$
For each $i \in \{ 1 \hdots n \}$ we define $C_i$ to equal the maximum $H^2$ norm of each coordinate of $W_i$ multiplied by $n$, then we apply Hölder's inequality on each coordinate to obtain
$$
||W_i \sum_{k=n}^{\infty} a_k^i z^k ||_{H^1(\mathbb{D}, \mathbb{C}^n)} \leqslant C_i || \sum_{k=n}^{\infty} a_k^i z^k ||_{H^2(\mathbb{D}, \mathbb{C}^n)} \to 0.
$$
Thus in the $H^1 ( \mathbb{D}, \mathbb{C}^n )$ norm we have $ a_n \to a_0$, a similar computation shows $\left(\beta_1 z + \hdots + \beta_n z^n \right)e_1(z)$ converges to $(\sum_{k=1}^{\infty}\beta_k z^k ) e_1$ in the $H^1 ( \mathbb{D}, \mathbb{C}^n )$ norm, so the $H^1 ( \mathbb{D}, \mathbb{C}^n )$ limit of $$ F(z) =  F_0(z) ( A_0 + A_1 z + \hdots A_{n-1} z^{n-1} )  + (\beta_1 z + \hdots + \beta_n z^n )e_1(z)$$ must be equal to
$$
F(z) = F_0 \left( \sum_{k=0}^{\infty} A_k z^k \right) + \left( \sum_{k=1}^{\infty} \beta_k z^k \right) e_1,
$$
and furthermore by taking limits in the equality of norms immediately after \eqref{star} we know
\begin{equation}\label{isom}
    ||F||^2 = \sum_{k=0}^{\infty}||A_k||^2 + \sum_{k=1}^{\infty}|\beta_k|^2 .
\end{equation}
We may alternatively express this as saying $F \in M$ if and only if 
\begin{equation}\label{start}
    F(z) = F_0 k_0 + z k_1 e_1 ,
\end{equation}
where $(k_0, k_1)$ lies in a subspace $K \subseteq H^2(\mathbb{D}, \mathbb{C}^r) \times H^2$ which is identified with $H^2(\mathbb{D}, \mathbb{C}^{r+1})$.

By virtue of \eqref{isom} we can see that $K$ is the image of a isometric mapping, and hence closed. We now argue $K$ is invariant under the backwards shift (on $H^2(\mathbb{D}, \mathbb{C}^{r+1})$). Since in the algorithm we have $k_0 (0) = A_0$ and $k_1(0) = \beta_1$ we can write $F$ as 
$$
F = F_0 A_0 + z F_0 S^*(k_0) + \beta_1 z e_1 + z^2 S^* (k_1) e_1 ,
$$
consequently
\begin{equation}\label{end}
    F_0 S^*(k_0) + z S^* (k_1) e_1 = \frac{F - F_0 A_0 - \beta_1 z e_1}{z} = G_1 \in M.
\end{equation}

Conversely if 
$$
M= \{ F_0 k_0 + z k_1 e_1: (k_0 , k_1) \in K \},
$$
is a closed subspace of $H^2(\mathbb{D}, \mathbb{C}^n)$, where $K$ is a $S^*$-invariant subspace of $H^2(\mathbb{D}, \mathbb{C}^{r+1})$, then $M$ is nearly $S^*$-invariant with defect 1. To show this we first need a lemma.

\begin{lem}\label{lemmaforlater}
$W_1(0), ... W_r(0)$ are linearly independent in $\mathbb{C}^n$.
\end{lem}
\begin{proof}
If $W_k(0) = \sum_{i \neq k} \lambda_i W_i(0)$ this would mean $W_k - \sum_{i \neq k} \lambda_i W_i$ vanishes at 0 and therefore lies in $zH^2(\mathbb{D}, \mathbb{C}^n )$.
\end{proof}
\noindent If $F \in M$ and $F(0)=0$ then we must have $F_0 (0) k_0 (0)$ is equal to the zero vector. We now add $n-r$ vectors $X_1 \hdots X_{n-r}$ which are linearly independent from $W_1(0), \hdots W_r(0)$ as extra columns to the matrix $F_0(0)$ to obtain a matrix
$$
F_0^{'}(0) = [W_1 , \hdots W_r, X_1 , \hdots, X_{n-r}].
$$
We now add $n-r$ extra $0$'s to the end of the column vector $k_0 (0)$ and label this $k_0^{'}(0)$. As $F_0 (0) k_0 (0)$ is equal to the zero vector, then $F_0^{'}(0) k_0^{'}(0)$ must also be equal to the zero vector. We can now invert $F_0^{'}(0)$ to obtain $k_0^{'}(0)$ is equal to the zero vector and hence $k_0(0)$ must be zero. This allows us to write 
$$
S^* (F) = F_0 \frac{k_0}{z} + k_1 e_1 = F_0 \frac{k_0}{z} + z S^* k_1 e_1 + k_1(0) e_1, 
$$
and as $K$ is $S^*$-invariant this is clearly an element of $M \oplus { \rm span}  \{ e_1 \} $.

\vskip 0.5cm

If all functions in $M$ vanish at $0$ then there is no non-trivial reproducing kernel at $0$, but we may now write 
$$
F(z) = z \left( G_1 (z) + \beta_1 e_1 (z) \right),
$$
with $G_1 \in M$ and $\beta_1 \in \mathbb{C}$, and furthermore 

$$
||F||^2 = ||G_1||^2 + |\beta_1|^2.
$$
We can then iterate on $G_1$ as we have previously done to obtain
$$
F(z) = \beta_1 z e_1 + \beta_2 z^2 e_1 + \hdots.
$$

For a general finite defect $m$ the analogous calculations produce the following result.

\begin{thm}\label{vector near invariant with defect}
Let $M$ be a closed subspace that is nearly $S^*$-invariant with a finite defect $m$. Then:
 \begin{enumerate}
     \item In the case where there are functions in $M$ that do not vanish at $0$,
     $$
     M = \{ F : F(z) = F_0 (z) k_0(z) + z \sum_{j=1}^{m} k_j(z) e_j (z) : ( k_0, \hdots ,k_m) \in K \} ,
     $$
     where $F_0$ is the matrix with each column being an orthonormal element of $W$, $\{ e_1 , \hdots e_m \}$ is any orthonormal basis for $D$, $k_0 \in H^2(\mathbb{D}, \mathbb{C}^r)$ (where $r = \dim W$), $k_1, \hdots k_m \in H^2$, and $K \subseteq H^2 ( \mathbb{D} , \mathbb{C}^{(r+m)})$ is a closed $S^*$-invariant subspace. Furthermore $ ||F||^2 = \sum_{j=0}^m ||k_j||^2$.
     \item In the case where all functions in $M$ vanish at $0$,
     $$
     M = \{ F : F(z) = z \sum_{j=1}^{m} k_j(z) e_j(z) : (k_1, \hdots, k_m) \in K \},
     $$
     with the same notation as in 1, except that $K$ is now a closed $S^*$-invariant subspace of $H^2( \mathbb{D}, \mathbb{C}^m )$, and $||F||^2 = \sum_{j=1}^{m} ||k_j||^2$.
 \end{enumerate}
 Conversely if a closed subspace $M \subseteq H^2(\mathbb{D}, \mathbb{C}^n)$ has a representation as in 1 or 2, the it is a nearly $S^*$-invariant subspace with defect $m$.
\end{thm}

\begin{rem}
The above Theorem was also independently proved in \cite{chattopadhyay2020invariant}.
\end{rem}

\section{Application to truncated Toeplitz operators}
Throughout this section our symbol $g$ is bounded and so the truncated Toeplitz operator $A_g^{\theta}: K_{\theta} \to K_{\theta}$ is defined by 
$$
A_g^{\theta} (f) = P_{\theta}(gf),
$$
where $P_{\theta}$ is the orthogonal projection $L^2 \to K_{\theta}$. 

It was originally observed in \cite{MR3398735} that the kernel of a truncated Toeplitz operator is the first coordinate of the kernel of the matricial Toeplitz operator with symbol 
$$
G = 
\begin{pmatrix}
\overline{\theta} & 0 \\
g & \theta \\
\end{pmatrix}.
$$
Scalar-type Toeplitz kernels (first introduced in \cite{scalartype}) are vector-valued Toeplitz kernels which can be expressed as the product of a space of scalar functions by a fixed vector function. A maximal function for $\ker T_G$ is an element $f \in \ker T_G$ such that if $f \in \ker T_H$ for any other bounded matricial symbol $H$, then $\ker T_G \subseteq \ker T_H$. By Corollary 3.9 in \cite{scalartype} $\ker T_G$ is of scalar type, it is also easily checked that $\ker T_G$ is not shift invariant and so by Theorem 3.7 in \cite{scalartype} we must have that $\ker T_G$ has a maximal function. Now by Theorem 3.10 of \cite{o2020toeplitz} we can deduce that $ W = \ker T_G \ominus (\ker T_G \cap z H^2( \mathbb{D},\mathbb{C}^n))$ has dimension 1. If we denote $\begin{pmatrix}
w_1 \\
w_2 \\
\end{pmatrix}
$
to be the normalised element of $W$ then using Corollary 4.5 from \cite{MR2651921} we can write 
$$
\ker T_G = \begin{pmatrix}
w_1 \\
w_2 \\
\end{pmatrix} K_{z \Phi},
$$
where $ \Phi $ is an inner function. We now can write
\begin{equation}\label{matrixtotruncated}
\ker A_g^{\theta} = w_1 K_{ z \Phi}.
\end{equation} We can describe $\Phi$ with the following proposition.
\begin{prop}\label{associated}
When $
\ker T_G = \begin{pmatrix}
w_1 \\
w_2 \\
\end{pmatrix} K_{z \Phi},
$ $\Phi $ is the unique (up to multiplication by a unimodular constant) inner function for which there exists $p_1, p_2 \in H^2$ such that 
$$
G  \begin{pmatrix}
w_1 \\
w_2 \\
\end{pmatrix}  \Phi = \begin{pmatrix}
\overline{z p_1} \\
\overline{z p_2} \\
\end{pmatrix},
$$
and $GCD( p_1^i , p_2^i ) = 1$.
\end{prop}
\begin{proof}
We first show that up to multiplication by a unitary constant there can only be one inner function $\Phi$ satisfying 
$$
G  \begin{pmatrix}
w_1 \\
w_2 \\
\end{pmatrix}  \Phi = \begin{pmatrix}
\overline{z p_1} \\
\overline{z p_2} \\
\end{pmatrix},
$$
where $GCD( p_1^i , p_2^i ) = 1$.
Suppose there are two inner functions $ \Phi_1 , \Phi_2$ such that
$$
G  \begin{pmatrix}
w_1 \\
w_2 \\
\end{pmatrix}  \Phi_1 = \begin{pmatrix}
\overline{z p_1} \\
\overline{z p_2} \\
\end{pmatrix},
$$
and 
$$
G  \begin{pmatrix}
w_1 \\
w_2 \\
\end{pmatrix}  \Phi_2 = \begin{pmatrix}
\overline{z q_1} \\
\overline{z q_2} \\
\end{pmatrix},
$$
where both $GCD( p_1^i , p_2^i ) = 1$ and $GCD( q_1^i , q_2^i ) = 1$. This would then imply that 
$$
\overline{\Phi_1} \begin{pmatrix}
\overline{z p_1} \\
\overline{z p_2} \\
\end{pmatrix} = \overline{\Phi_2}  \begin{pmatrix}
\overline{z q_1} \\
\overline{z q_2} \\
\end{pmatrix},
$$
and so $(\Phi_1 p_1)^i =(\Phi_2 q_1)^i$ and $(\Phi_1 p_2)^i =(\Phi_2 q_2)^i$. By assumption we have $GCD( p_1^i , p_2^i )$ $= 1$ so $GCD( (\Phi_1 p_2)^i , (\Phi_1 p_1)^i ) = \Phi_1$, but substituting $(\Phi_1 p_1)^i$ for $(\Phi_2 q_1)^i$ we obtain $$GCD ( (\Phi_1 p_2)^i , (\Phi_2 q_1)^i )= \Phi_1, $$ and so $\Phi_1 $ divides $\Phi_2$. A similar computation shows $\Phi_2$ divides $\Phi_1$, and so we must have $\Phi_1$ is a unitary constant multiple of $\Phi_2$. We now show that $\Phi$ is such that 
$$
G  \begin{pmatrix}
w_1 \\
w_2 \\
\end{pmatrix}  \Phi = \begin{pmatrix}
\overline{z p_1} \\
\overline{z p_2} \\
\end{pmatrix},
$$
with $GCD( p_1^i , p_2^i ) = 1$. If it is the case that $ \alpha = GCD( p_1^i , p_2^i ) \neq 1$ then it would follow that $\begin{pmatrix}
w_1 \\
w_2 \\
\end{pmatrix}  \Phi \alpha \in \ker T_G$, which would be a contradiction as $\Phi \alpha \notin K_{z \Phi}$.
\end{proof}

It is easily checked that $\ker T_G$ is nearly $S^*$-invariant, and because $\ker A_g^{\theta} = P_1 (\ker T_G )$ we can use Corollary \ref{main2} to deduce the kernel of a truncated Toeplitz operator is nearly $S^*$-invariant with a defect given by ${ \rm span} \{\frac{w_1}{z}\} \cap H^2$. With this information we can use the following  result given as Theorem 3.2 in \cite{chalendar2019beurling} (or equivalently Theorem \ref{vector near invariant with defect} with $n=1$) to study $\ker A_g^{\theta}$.
\begin{thm}\label{4.2}
Let $M \subseteq H^2$ be a closed subspace that is nearly $S^*$-invariant with a finite defect $m$. Then:
 \begin{enumerate}
     \item In the case where there are functions in $M$ that do not vanish at $0$,
     $$
     M = \{ f : f(z) = f_0 (z) k_0(z) + z \sum_{j=1}^{m} k_j(z) e_j (z) : ( k_0, \hdots ,k_m) \in K \} ,
     $$
     where $f_0$ is the normalised reproducing kernel for $M$ at $0$, $\{ e_1 , \hdots e_m \}$ is any orthonormal basis for $D$, and $K$ is a closed $S^*$-invariant subspace of $H^2 ( \mathbb{D} , \mathbb{C}^{(m+1)})$. Furthermore $ ||f||^2 = \sum_{j=0}^m ||k_j||^2$.
     \item In the case where all functions in $M$ vanish at $0$,
     $$
     M = \{ f : f(z) = z \sum_{j=1}^{m} k_j(z) e_j(z) : (k_1, \hdots, k_m) \in K \},
     $$
     with the same notation as in 1, except that $K$ is now a closed $S^*$-invariant subspace of $H^2( \mathbb{D}, \mathbb{C}^m )$, and $||f||^2 = \sum_{j=1}^{m} ||k_j||^2$.
 \end{enumerate}
 Conversely if a closed subspace $M \subseteq H^2$ has a representation as in 1 or 2, the it is a nearly $S^*$-invariant subspace with defect $m$.
\end{thm}

To use Theorem \ref{4.2} we have to assume that our defect space is orthogonal to $\ker A_g^{\theta}$; we consider two separate cases. We first assume that all functions in $\ker A_g^{\theta}$ vanish at 0. We set $O := \ker A_g^{\theta} + { \rm span} \{\frac{w_1}{z}\} $, $E := O \ominus \ker A_g^{\theta}$, we let $e$ be $P_{E}(\frac{w_1}{z})$ and then $e$ is orthogonal to $\ker A_g^{\theta}$. In this construction $e \neq 0$ as this would imply $\frac{w_1}{z} \in \ker A_g^{\theta} = w_1 K_{ z \Phi}$ which is clearly a contradiction. Theorem \ref{4.2} now yields
$$
\ker A_g^{\theta} = e z K_{ \Psi},
$$
where multiplication by $e z$ is an isometry from $K_{\Psi}$ to $\ker A_g^{\theta}$. This expression for $\ker A_g^{\theta}$ is more familiar than $w_1 K_{z \Phi} $ as in this case the multiplication is an isometry as opposed to a contraction. We can also relate this expression to nearly $S^*$-invariant subspaces. If we let $n$ be the greatest natural number such that $\frac{e}{z^n} \in H^2$ then $\frac{\ker A_g^{\theta}}{z^{n+1}} = \frac{e}{z^n}  K_{z \Psi}$, now $\frac{e}{z^n}(0) \neq 0 $ so $\frac{\ker A_g^{\theta}}{z^{n+1}} = \frac{e}{z^n}  K_{z \Psi}$ is a nearly $S^*$-invariant subspace. We can conclude the following theorem in this case.

\begin{thm}
If $n$ is the greatest natural number such that $\ker A_g^{\theta} \subseteq z^n H^2$, then $\frac{ \ker A_g^{\theta}}{z^n}$ is a nearly $S^*$-invariant subspace.
\end{thm}

We now turn our attention to the case when not all functions in $\ker A_g^{\theta}$ vanish at 0. In this case it must also follow that $w_1(0) \neq 0$ as otherwise $w_1 K_{z \Phi} (0) = 0$, so using Corollary \ref{main2} we must have the defect space for $\ker A_g^{\theta} = 0$ so can conclude the following theorem.

\begin{thm}
If $\ker A_g^{\theta}$ contains functions which do not vanish at 0 then it is nearly $S^*$-invariant.
\end{thm}

When $\ker A_g^{\theta}$ is nearly $S^*$-invariant we may proceed by using Proposition 3 of the paper of Hitt \cite{hitt1988invariant} to show $\ker A_g^{\theta} = u K_{z \psi}$ where $u  \in \ker A_g^{\theta} \ominus ( \ker A_g^{\theta} \cap z H^2 )$ is an isometric multiplier. As was noted in \cite{hartmann2003extremal} we can call $\psi$ the associated inner function to $u$, and it is easily checked (similar to the approach in Proposition \ref{associated}) this is an inner function such that $ g u \psi = \overline{ z p_1 } + \theta p_2$ where $p_1$ is outer.

In fact using \eqref{matrixtotruncated} we can view these two theorems as specialisations of the following theorem.

\begin{thm}\label{twoways}
If $f \in H^2$, and $I$ is an inner function such that $f K_I$ is a closed subspace of $H^2$, then if $f(0) \neq 0$ then $f K_I$ is a nearly invariant subspace. If $f(0) = 0$ then $f K_I$ is both a nearly invariant subspace multiplied by a power of $z$ and a nearly invariant subspace with a 1-dimensional defect space $\frac{f}{z} ( K_I \ominus (K_I \cap z H^2))$.
\end{thm}

\begin{proof}
The only non-trivial statement to prove is if $f(0) = 0$ then $f K_I$ is a nearly invariant subspace with a defect space $\frac{f}{z} ( K_I \ominus (K_I \cap z H^2))$, but this follows from
$$
\frac{f K_I}{z} \in \frac{f}{z}( K_I \ominus (K_I \cap z H^2)) + f ( \frac{K_I \cap z H^2}{z}) \subseteq \frac{f}{z}( K_I \ominus (K_I \cap z H^2)) + f K_I.
$$
\end{proof}

So under the assumptions $f \in H^2$ and $I$ is an inner function such that $f K_I$ is a closed subspace of $H^2$, if $f(0) = 0$ then Theorem \ref{twoways} gives us two possible approaches to decomposing $f K_{I}$. \begin{enumerate}
    \item Divide $f K_{I}$ by $z^n$ where $n \in \mathbb{N}$ is chosen such that $\frac{f}{z^n} (0) \neq 0$, then use the Hitt decomposition given in \cite{hitt1988invariant}. Then we could write $f K_I$ as $z^n u$ multiplied by some model space, where $u \in \frac{f K_I}{z^n} \ominus ( \frac{f K_I}{z^n} \cap z H^2 ) $ .
    \item Use Theorem 3.2 in \cite{chalendar2019beurling} with $\frac{f}{z}( K_I \ominus (K_I \cap z H^2)) $ as the defect space. Then we could write $f K_I$ as $ z e $ multiplied by some model space , where $e$ is chosen to be an element of $\frac{f}{z}( K_I \ominus (K_I \cap z H^2)) + f K_I$ orthogonal to $ f K_I$.
\end{enumerate} In both of these cases we obtain a model space multiplied by an isometric multiplier. 

Due to the similarities in the way these two decompositions are developed one might expect that the two possible ways of decomposing $f K_{I}$ might actually yield the same result. We show this is not the case and in general we have two different expressions with an example.

\begin{exmp}
Let $g = \frac{1}{1 - \frac{z}{3}} (\overline{z}^3 + z^3)$ and let $\theta = z^4$, we first find $ \ker A_g^{\theta}$ using linear algebra techniques. With respect to the basis $  1, z, z^2, z^3 $, $ A_g^{\theta}$ has the matrix representation 
$$
\begin{pmatrix}
\frac{1}{3^3} & \frac{1}{3^2} & \frac{1}{3} & 1 \\
\frac{1}{3^4} & \frac{1}{3^3} & \frac{1}{3^2} & \frac{1}{3} \\
\frac{1}{3^5} & \frac{1}{3^4} & \frac{1}{3^3} & \frac{1}{3^2} \\
1 + \frac{1}{3^6} & \frac{1}{3^5} & \frac{1}{3^4} & \frac{1}{3^3} \\
\end{pmatrix},
$$
which has reduced row echelon form given by 
$$
\begin{pmatrix}
1 & 0 & 0 & 0 \\
0 & 1 & 3 & 9 \\
0 & 0 & 0 & 0 \\
0 & 0 & 0 & 0 \\
\end{pmatrix}.
$$
The kernel of this matrix has a basis given by 
$$
\begin{pmatrix}
0 \\
1 \\
- \frac{1}{3} \\
0 \\
\end{pmatrix}, \begin{pmatrix}
0 \\
0 \\
1 \\
- \frac{1}{3} \\
\end{pmatrix},
$$
and thus we can write $\ker A_g^{\theta} = z ( 1 - \frac{z}{3}) K_{z^2}$. We now will give two different decompositions of this kernel using Theorem \ref{twoways}.
Let $f = z ( 1 - \frac{z}{3})$ and $K_{I} = K_{z^2}$, then $f K_{I} = z { \rm span} \{( 1 - \frac{z}{3}) , z ( 1 - \frac{z}{3}) \}$. We first use approach 1. It can be checked that 
$$
  1 - \frac{z}{3} K_{z^2} \ominus  1 - \frac{z}{3} K_{z^2} \cap z H^2 
$$
has a normalised basis element given by $$ u = \frac{3 \sqrt{910}}{91}(1 -  \frac{ 1}{30} z  - \frac{1}{10} z^2) , $$ and so $f K_{I}$ can be written as $z u $ multiplied by some model space, which we will denote $K_{I_1}$. In order to find $I_1$ we must solve
$$
z ( 1 - \frac{z}{3}) K_{z^2} = z u K_{I_1},
$$
but $ \frac{( 1 - \frac{z}{3})}{u}$ is a scalar multiple of $\frac{1}{1 + \frac{3z}{10}}$, so $K_{I_1}$ must be given by ${ \rm span} \{ \frac{1}{1 + \frac{3z}{10}}  \frac{z}{1 + \frac{3z}{10}} \}$, therefore $I_1 = z\frac{z+ \frac{3}{10}}{1 + \frac{3z}{10}}$. So we conclude
$$
z ( 1 - \frac{z}{3}) K_{z^2} =  z \frac{3 \sqrt{910}}{91}(1 -  \frac{1}{30} z  - \frac{1}{10} z^2 ) K_{ z(\frac{z+ \frac{3}{10}}{1 + \frac{3z}{10}})},
$$
where multiplication by $z \frac{3 \sqrt{910}}{91}(1 -  \frac{1}{30} z  - \frac{1}{10} z^2 )$ is an isometry on the model space. This can be simplified to 
$$
z ( 1 - \frac{z}{3}) K_{z^2} = z ( 30 - z - 3 z^2) K_{z(\frac{z+ \frac{1}{3}}{1 + \frac{z}{3}})},
$$
however in this case we no longer have the multiplication on the model space acting as an isometry.
Now we use approach 2. We must find a normalised element $e \in z ( 1 - \frac{z}{3}) K_{z^2} + { \rm span} \{  ( 1 - \frac{z}{3}) \}$, which is orthogonal to $z  ( 1 - \frac{z}{3}) K_{z^2}$. This can be checked to be
$$
\sqrt{\frac{729}{74620}} (\frac{91}{9} - \frac{1}{27} z - \frac{1}{9}z^2 - \frac{1}{3}z^3 )
$$
which means $f K_I$ can also be written as $ze$ multiplied by some model space, which we will denote $K_{I_2}$. Now to find $I_2$ we must solve
$$
z ( 1 - \frac{z}{3}) K_{z^2} = z e K_{I_2},
$$
$e$ is a scalar multiple of $$
( 273 - z - 3z^2 - 9 z^3 ) = 3 ( 1 - \frac{z}{3}) ( 9 z^2 + 30z + 91 ), $$ and so $K_{I_2}$ must be ${ \rm span} \{ \frac{1}{9z^2 +30z + 91} , \frac{z}{9z^2 +30z + 91} \}$. We now aim to find the inner function $I_2$. We denote $A = \frac{1}{9z^2 +30z + 91}$ and $B = \frac{z}{9z^2 +30z + 91}$. $A(0) = \frac{1}{91}$, so $$S^* (A)(z) = \frac{A(z) - \frac{1}{91}}{z} = \frac{-9z - 30}{91 (9 z^2 + 30 z + 91)} = -\frac{30}{91} A - \frac{9}{91} B. $$ It is clear that $S^* (B) = A$. We now aim to find two eigenvectors of the backwards shift operator (these are necessarily Cauchy kernels) which are in ${ \rm span} \{ A , B \} $. If we use $ A , B $ as a basis for ${ \rm span} \{ A , B \}$ then the matrix representation of the backwards shift operator is given by 
$$
\begin{pmatrix}
-\frac{30}{91} & 1 \\
- \frac{9}{91} & 0 \\
\end{pmatrix}.
$$
This has eigenvalues given by $\frac{-15 \pm 3i \sqrt{66}}{91}$, we denote $\lambda_1 = \frac{-15 + 3i \sqrt{66}}{91}$ and $\lambda_2 = \frac{-15 - 3i \sqrt{66}}{91}$, then the corresponding eigenvectors are given by $k_{\overline{\lambda_1}} = \frac{1}{1-\lambda_1 z}$ and $k_{\overline{\lambda_2}} = \frac{1}{1-\lambda_2 z}$, so as mentioned in the introduction $I_2 = (\frac{z- \overline{\lambda_1}}{1 - {\lambda_1}z})(\frac{z- \overline{\lambda_2}}{1 - {\lambda_2 }z})$. We can conclude 
$$
z ( 1 - \frac{z}{3}) K_{z^2} =  z \sqrt{\frac{729}{74620}} (\frac{91}{9} - \frac{1}{27} z - \frac{1}{9}z^2 - \frac{1}{3}z^3) K_{(\frac{z- \overline{\lambda_1}}{1 - {\lambda_1}z})(\frac{z- \overline{\lambda_2}}{1 - {\lambda_2 }z})},
$$
where multiplication by $z \sqrt{\frac{729}{74620}} (\frac{91}{9} - \frac{1}{27} z - \frac{1}{9}z^2 - \frac{1}{3}z^3)$ is an isometry on the model space. Again we can simplify this to 
$$
z ( 1 - \frac{z}{3}) K_{z^2} =  z (273 - z - 3z^2 - 9z^3) K_{(\frac{z- \overline{\lambda_1}}{1 - {\lambda_1}z})(\frac{z- \overline{\lambda_2}}{1 - {\lambda_2 }z})},
$$
but in this expression we no longer have the multiplication on the model space acting as an isometry.
Thus approach 1 and approach 2 give different decompositions.

\end{exmp}

\section{Application to truncated Toeplitz operators on multiband spaces}
Truncated Toeplitz operators on multiband spaces are soon to be introduced in a publication which is currently in preparation by M.C. Câmara, R. O’Loughlin, and J.R. Partington. They are defined (on the unit circle) as follows. Let $g \in L^{\infty}$, let $\phi, \psi$ be unimodular functions in $L^{\infty}$ such that $\phi K_{\theta} \perp \psi K_{\theta} $, we define the multiband space $M := \phi K_{\theta}  \oplus  \psi K_{\theta}$. The truncated Toeplitz operator on $M$ denoted $A_g^{M} : M \to M$ is defined by 
$$
A_g^M(f) = P_M ( gf),
$$
where $P_M$ is the orthogonal projection on to $M$. These operators have applications in speech processing and as a special case, if we let $\phi = \overline{\theta}$ and $  \psi  = \theta$ we recover a (disc variation) of the Paley-Wiener space.

We write $K_{\theta}(\mathbb{D}, \mathbb{C}^n) \subseteq H^2(\mathbb{D},\mathbb{C}^n)$ to mean the vectors of length $n$ with each coordinate taking entries in $K_{\theta}$. To study truncated Toeplitz operators on multiband spaces we first consider the truncated Toeplitz operator $A_G^{\theta}$ acting on $K_{\theta}(\mathbb{D},\mathbb{C}^2)$, where 
$$
G =\begin{pmatrix}
g_{11} & g_{12} \\
g_{21} & g_{22} \\
\end{pmatrix},
$$
has each entry in $L^{\infty}$. Using the unitary map $U: M \to K_{\theta}(\mathbb{D}, \mathbb{C}^2)$ where $$U(\phi f_1 +  \psi f_2 ) = \begin{pmatrix}
f_1 \\
f_2
\end{pmatrix},$$ one can show that any truncated Toeplitz operator on a multiband space is unitarily equivalent to $A_G^{\theta}$ for a certain choice of $G$. Thus we turn our attention to studying $ \ker A_{G}^{\theta}$.

If we define 
$$
\mathcal{G} = \begin{pmatrix}
\overline{\theta} & 0 & 0 & 0 \\
0 & \overline{\theta} & 0 & 0 \\
g_{11} & g_{12} & \theta & 0 \\
g_{21} & g_{22} & 0 & \theta \\
\end{pmatrix},
$$
then it is easily checked that 
$$
\begin{pmatrix}
p \\
q \\
r \\
s \\
\end{pmatrix} \in \ker T_{\mathcal{G}}
$$
if and only if $p, q \in K_{\theta} $ and 
$$
\begin{pmatrix}
g_{11} & g_{12} \\
g_{21} & g_{22} \\
\end{pmatrix} \begin{pmatrix}
p \\
q \\
\end{pmatrix} + \theta \begin{pmatrix}
r \\
s \\
\end{pmatrix} \in \overline{H^2_0} \oplus \overline{H^2_0}. 
$$
So $\begin{pmatrix}
p \\
q \\
\end{pmatrix} \in \ker A_{G}^{\theta}$, and likewise given $\begin{pmatrix}
p \\
q \\
\end{pmatrix} \in \ker A_{G}^{\theta}$ there exists $\begin{pmatrix}
r \\
s \\
\end{pmatrix} \in H^2$ with $\begin{pmatrix}
p \\
q \\
r \\
s \\
\end{pmatrix} \in \ker T_{\mathcal{G}}$. Keeping the same notation as Theorem \ref{main} we let $W = \ker T_{\mathcal{G}} \ominus ( \ker T_{\mathcal{G}} \cap z H^2(\mathbb{D},\mathbb{C}^4))$, and let $W_1 \hdots W_r$ be an orthonormal basis for $W$, as previously mentioned $r \leqslant 4$. Toeplitz kernels are nearly $S^*$-invariant so by Corollary \ref{main2} we know $P_{2}(\ker T_{\mathcal{G}}) = \ker A_{G}^{\theta}$ is nearly $S^*$-invariant with defect space $\left( \frac{ { \rm span} \{ { P_{2}(W_1), \hdots P_{2}(W_r)} \} }{z} \cap H^2(\mathbb{D}, \mathbb{C}^2)\right)$. We now try to find the dimension of this defect space. For $F$ a set of functions we denote
$$
F(0) = \{ f(0) : f \in F \}.
$$
\begin{lem}\label{evaluating at 0}
$\dim \ker T_{\mathcal{G}}(0) = \dim W = \dim W(0) $.
\end{lem}
\begin{proof}
Lemma 3.9 in \cite{o2020toeplitz} shows that $\dim \ker T_{\mathcal{G}}(0) = \dim W$, Lemma \ref{lemmaforlater} shows that $W_1(0) \hdots W_r(0)$ are linearly independent and clearly $W_1(0) \hdots W_r(0)$ { \rm span} $W(0)$.
\end{proof}

We first consider the case when $\dim W = 4$, in this case by Lemma \ref{evaluating at 0} we have $W(0) = \mathbb{C}^4$. We have a correspondence between the matrix $[ W_1, W_2, W_3, W_4]$ and a $4$-by-$4$ matrix taking values in $\mathbb{C}$ given by
$$
[ W_1, W_2, W_3, W_4] \mapsto [ W_1(0), W_2(0), W_3(0), W_4(0)].
$$
We also know by Lemma \ref{evaluating at 0} that $ W_1(0), W_2(0), W_3(0), W_4(0)$ are a basis for $\mathbb{C}^4$, so there exists a sequence of column operations we can perform to $ W_1(0), W_2(0), W_3(0),$ $W_4(0)$ which yields the identity matrix. If we perform the same column operations to $[ W_1, W_2, W_3, W_4]$ we will obtain a matrix 
$$
[ \tilde{W_1}, \tilde{W_2}, \tilde{W_3}, \tilde{W_4}],
$$
which has the property that $[ \tilde{W_1}(0), \tilde{W_2}(0), \tilde{W_3}(0), \tilde{W_4}(0)]$ is equal to the identity matrix. The linear independence of $\tilde{W_1}(0), \tilde{W_2}(0), \tilde{W_3}(0), \tilde{W_4}(0)$ implies linear independence of $ \tilde{W_1}, \tilde{W_2}, \tilde{W_3}, \tilde{W_4}$, and so $ \tilde{W_1}, \tilde{W_2}, \tilde{W_3}, \tilde{W_4}$ { \rm span} $W$. It is now clear that $\left( \frac{ { \rm span} \{ { P_{2}(W_1), \hdots P_{2}(W_4)} \} }{z} \cap H^2(\mathbb{D}, \mathbb{C}^2) \right)$ is given by $\frac{ { \rm span} \{ { P_{2}(\tilde{W_3}) P_{2}(\tilde{W_4})} \} }{z}$, and so when $\dim W = 4$, we have $\ker A_{g}^{\theta}$ is nearly invariant with the 2-dimensional defect space $\frac{ { \rm span} \{ { P_{2}(\tilde{W_3}) P_{2}(\tilde{W_4})} \} }{z}$.

We now consider the case when $\dim W =3$, in this case $W(0)$ is a 3-dimensional subspace of $\mathbb{C}^4$. We again have a correspondence 
$$
[ W_1, W_2, W_3] \mapsto [ W_1(0), W_2(0), W_3(0)].
$$
In this case we can perform column operations to $ W_1(0), W_2(0), W_3(0)$ to obtain a matrix which takes one of the following four forms (here we denote $x_1, x_2, x_3$ to be some unknown unspecified values in $\mathbb{C}$),
$$
    \begin{pmatrix}
    1 & 0 & 0 \\
    0 & 1 & 0 \\
    0 & 0 & 1 \\
    x_1 & x_2 & x_3 \\
    \end{pmatrix},
    \begin{pmatrix}
    1 & 0 & 0 \\
    0 & 1 & 0 \\
    x_1 & x_2 & x_3 \\
    0 & 0 & 1 \\
    \end{pmatrix},
    \begin{pmatrix}
    1 & 0 & 0 \\
    x_1 & x_2 & x_3 \\
    0 & 1 & 0 \\
    0 & 0 & 1 \\
    \end{pmatrix},
   \begin{pmatrix}
    x_1 & x_2 & x_3 \\
    1 & 0 & 0 \\
    0 & 1 & 0 \\
    0 & 0 & 1 \\
    \end{pmatrix}.
    $$
As in the previous case if we perform these same column operations which yield one of the above to the matrix $[ W_1, W_2, W_3 ]$ we will obtain a matrix 
$$
[ \tilde{W_1}, \tilde{W_2}, \tilde{W_3}].
$$
By the same arguments made previously we can deduce $\tilde{W_1}, \tilde{W_2}, \tilde{W_3}$ { \rm span} $W$. This means $\left( \frac{ { \rm span} \{ { P_{2}(W_1), \hdots P_{2}(W_3)} \} }{z} \cap H^2 \right)$ is contained in $\frac{ { \rm span} \{ { P_{2}(\tilde{W_2}) P_{2}(\tilde{W_3})} \} }{z} \cap H^2(\mathbb{D}, \mathbb{C}^2)$, and so when $\dim W = 3$, we have $\ker A_{G}^{\theta}$ is nearly invariant with (at most) 2-dimensional defect space $$\frac{ { \rm span} \{ { P_{2}(\tilde{W_2}) P_{2}(\tilde{W_3})} \} }{z} \cap H^2(\mathbb{D}, \mathbb{C}^2).$$

In the case when $\dim W \leqslant 2$ it is clear from Corollary \ref{main2} that the defect space of $\ker A_{G}^{\theta}$ has dimension at most 2. So we can conclude the following theorem.

\begin{thm}\label{multiband defect}
$\ker A_{G}^{\theta}$ is a nearly $S^*$-invariant subspace with defect 2.
\end{thm}
We now give an example to show that in general $2$ is the smallest dimension of defect space, i.e. it is not true that for all inner functions $\theta$ and matrix symbols $G$ that $\ker A_{G}^{\theta}$ has a 1-dimensional defect.

\begin{exmp} Following the unitary equivalence we mentioned earlier we consider an operator of the form
$$
A_{G}^{\theta} = \begin{pmatrix}
A_{g}^{\theta} & A_{g \overline{\phi} \psi}^{\theta} \\
A_{g \overline{\psi} \phi}^{\theta} & A_g^{\theta} \\
\end{pmatrix},
$$
where $g \in L^{\infty}$, $\theta$ is an inner function and $\phi , \psi \in L^{\infty}$ are unimodular functions such that $\phi K_{\theta} \perp \psi K_{\theta}$. These conditions ensure that $A_G^{\theta}$ is indeed unitarily equivalent to a truncated Toeplitz operator a multiband space. Let $\theta = z^2$, $\phi = z$, $\psi = z^4$, $g = 2 \overline{z}^2 + z + 2 z^4$. We identify the basis of $K_{\theta}(\mathbb{D}, \mathbb{C}^2)$ with a basis of $\mathbb{C}^4$ in the following way $\begin{pmatrix}
1 \\
0 \\
\end{pmatrix} \mapsto \begin{pmatrix}
1 \\
0 \\
0 \\
0 \\
\end{pmatrix} $,$\begin{pmatrix}
z \\
0 \\
\end{pmatrix} \mapsto \begin{pmatrix}
0 \\
1 \\
0 \\
0 \\
\end{pmatrix} $, $\begin{pmatrix}
0 \\
1 \\
\end{pmatrix} \mapsto \begin{pmatrix}
0 \\
0 \\
1 \\
0 \\
\end{pmatrix} $, $\begin{pmatrix}
0 \\
z \\
\end{pmatrix} \mapsto \begin{pmatrix}
0 \\
0 \\
0 \\
1 \\
\end{pmatrix} $, then $A_G^{\theta}$ has the following matrix representation 
$$
\begin{pmatrix}
0 & 0 & 0 & 0 \\
1 & 0 & 2 & 0 \\
0 & 0 & 0 & 0 \\
2 & 0 & 1 & 0 \\
\end{pmatrix}.
$$
Thus $\ker A_{G}^{\theta} $ is given by the span of $\begin{pmatrix}
z \\
0 \\
\end{pmatrix}$ and $\begin{pmatrix}
0 \\
z \\
\end{pmatrix}$, which is clearly nearly $S^*$-invariant with defect 2.
\end{exmp}
For a multiband space $M := \phi K_{\theta}  \oplus  \psi K_{\theta}$ using our unitary equivalence by $U$ we can write 
$$
\ker A_g^{M} = U^* \ker  \begin{pmatrix}
A_{g}^{\theta} & A_{g \overline{\phi} \psi}^{\theta} \\
A_{g \overline{\psi} \phi}^{\theta} & A_g^{\theta} \\
\end{pmatrix}.
$$
Combining this with Theorem \ref{multiband defect} and Theorem \ref{vector near invariant with defect} gives a decomposition for $\ker A_g^M$ in terms of $S^*$-invariant subspaces.
\section{Application to dual truncated Toeplitz operators}
It is easily checked that in $L^2$ we have $K_{\theta}^{\perp} =\overline{H^2_0} \oplus  \theta H^2$. We denote $Q$ to be the orthogonal projection $Q: L^2 \to (K_{\theta})^{\perp}$. Throughout this section we assume $g \in L^{\infty}$. The dual truncated Toeplitz operator $D_{g}^{\theta} : (K_{\theta})^{\perp} \to (K_{\theta})^{\perp}$ is defined by 
$$
f \mapsto Q(gf).
$$
Theorem 6.6 in \cite{camara2019invertibility} shows that for a symbol $g$ that is invertible in $L^{\infty}$ we have $\ker D_{g}^{\theta} = g^{-1} \ker A_{g^{-1}}^{\theta}$, so given our observation \eqref{matrixtotruncated} under the condition that $g$ is invertible in $L^{\infty}$ we can write $\ker D_{g}^{\theta}$ as an $L^2$ function multiplied by a model space. We now aim to use similar recursive methods that were used to prove Theorem \ref{vector near invariant with defect} to obtain a decomposition theorem for $\ker D_{g}^{\theta}$.

\vskip 0.4cm
Throughout this section we assume that $\ker D_{g}^{\theta}$ is finite dimensional.

\vskip 0.4cm
We define $A := \{ f \in \ker D_{g}^{\theta} : gf \in K_{\theta} \cap z H^2 \}$ and $C := \ker D_g^{\theta} \cap  ( \overline{H^2_0} \oplus \theta z H^2 ) \cap A $, then using orthogonal decomposition we can write 
$$
\ker D_g^{\theta} = C \oplus ( \ker D_{g}^{\theta} \ominus C ).
$$
\begin{lem}\label{lem1}
If $\ker D_g^{\theta} \subseteq C$ then $\ker D_g^{\theta} = \{ 0 \}$.
\end{lem}
\begin{proof}
Suppose we have a non-zero $ f \in \ker D_g^{\theta} \subseteq C $, then by construction of $C$ we must have $\frac{f}{z} \in \ker D_g^{\theta} \subseteq C$. Iterating this we can obtain $\frac{f}{z^n} \in \ker D_g^{\theta}$ for all $ n \in \mathbb{N}$, which can't be true as given $n$ sufficiently large $\frac{gf}{z^n} \notin H^2$. 
\end{proof}

\begin{cor}\label{cor2}
For any $\ker D_g^{\theta} \neq \{ 0 \} $ we have $1 \leqslant \dim (\ker D_{g}^{\theta} \ominus C ) \leqslant 2 $. 
\end{cor}
\begin{proof}
If $\ker D_g^{\theta} \neq \{ 0 \} $ then Lemma \ref{lem1} shows that $1 \leqslant \dim (\ker D_{g}^{\theta} \ominus C )$. Let $F_1 $ be the orthogonal projection of $\overline{g} k_0$ on to $\ker D_g^{\theta}$ and $F_2 $ be the orthogonal projection of $\theta k_0$ on to $\ker D_g^{\theta}$, where $k_0$ is the reproducing kernel at 0, then $\ker D_g^{\theta} \ominus C$ is generated by $F_1, F_2$. Indeed if $f \in \ker D_g^{\theta}$ and $f$ is orthogonal to $F_1, F_2$ then 
$$
\langle f , F_1 \rangle = \langle gf, k_0 \rangle = 0,
$$
so $f \in A$, and 
$$
\langle f , F_2 \rangle = \langle \overline{\theta}f, k_0 \rangle = 0,
$$
so we also have $P( \overline{\theta}f) \subseteq z H^2$, so $f \in C$.
\end{proof}

Consider $g\ker D_g^{\theta} = gC \oplus ( g\ker D_g^{\theta} \ominus gC ) $, by Corollary \ref{cor2} we must have $g \ker D_g^{\theta} \ominus gC$ is at most 2-dimensional. If $g \ker D_g^{\theta} \ominus gC$ is 2-dimensional then we denote its orthonormal basis elements by $gf_0, gh_0 $. Then for all $f \in \ker D_g^{\theta}$ using orthogonal projections and the observation that $\frac{C}{z} \subseteq \ker D_g^{\theta}$ we can write
$$
gf = \lambda_0 g f_0 + \mu_0 g h_0 + z gf_1,
$$
where $gf_1 \in g\ker D_g^{\theta}$, and furthermore
$$
|| gf ||^2 = | \lambda_0 |^2 + | \mu_0 |^2 + ||gf_1 ||^2.
$$
In a similar process to Theorem \ref{vector near invariant with defect} we can iterate this process starting with $gf_1$ to obtain
$$
gf = \sum_{i=0}^N gf_0 \lambda_i z^i + \sum_{j=0}^{N} gh_0 \mu_j z^j + z^{N+1} gf_{N+1},
$$
with 
\begin{equation}\label{here}
    ||gf||^2 = \sum_{i=0}^N |\lambda_i|^2 + \sum_{j=0}^N |\mu_j|^2 + ||gf_{N+1}||.
\end{equation}
Following the argument laid out in section 3 to deduce \eqref{limit} we can deduce that in the $H^2$ norm $|| g f_{N+1} || \to 0$ as $N \to \infty$, then $|| g f_{N+1} ||$ must also converge to $0$ in the $L^1$ norm , and so in the $L^1$ norm we must have 
$$
gf = \lim_{N \to \infty} \left( \sum_{i=0}^N gf_0 \lambda_i z^i + \sum_{j=0}^{N} gh_0 \mu_j z^j \right).
$$
Now two applications of Hölder's inequality shows the $L^1$ limit of $\sum_{i=0}^N gf_0 \lambda_i z^i + \sum_{j=0}^{N} gh_0 \mu_j z^j$ is equal to $g f_0  \sum_{i=0}^{\infty} \lambda_i z^i + gh_0 \sum_{j=0}^{\infty}  \mu_j z^j$, where $\sum_{i=0}^{\infty} \lambda_i z^i , \sum_{j=0}^{\infty}  \mu_j z^j$ are limits in the $H^2$ sense . So we may write
$$
gf = g f_0  \sum_{i=0}^{\infty}  \lambda_i z^i + gh_0 \sum_{j=0}^{\infty}  \mu_j z^j,
$$
and furthermore by taking  limits in \eqref{here} we can deduce
$$
||gf||_{H^2}^2 = \sum_{i=0}^{\infty} |\lambda_i|^2 +  \sum_{i=0}^{\infty} |\mu_i|^2.
$$
Mimicking the argument from section 3 between \eqref{start} and \eqref{end} we can say $f \in \ker D_g^{\theta}$ if and only if
$$
gf = 
\begin{pmatrix}
gf_0 & gh_0
\end{pmatrix}
\begin{pmatrix}
k_0 \\
k_1 \\
\end{pmatrix},
$$
where $\begin{pmatrix}
k_0 \\
k_1 \\
\end{pmatrix}$ lies in a closed $S^*$-invariant subspace of $H^2 ( \mathbb{D} , \mathbb{C}^2)$. With obvious modifications for when $\dim \ker D_g^{\theta} \ominus C = 1$ we can deduce the following theorem.
\begin{thm}
\begin{enumerate}
\item If $\dim (g \ker D_g^{\theta} \ominus gC) = 2$ then 
$$
g \ker D_g^{\theta} = \begin{pmatrix} gf_0 & gh_0 \\ \end{pmatrix}
K,$$
where $K$ is a closed $S^*$-invariant subspace of $H^2( \mathbb{D}, \mathbb{C}^2)$, $gf_0, gh_0$ are orthonormal basis elements of $(g \ker D_g^{\theta} \ominus gC)$ and for $f \in \ker D_g^{\theta}$ we have $||gf||_{H^2}^2 = ||k_0||_{H^2}^2 + || k_1 ||_{H^2}^2$.
\item If  $\dim (g \ker D_g^{\theta} \ominus gC) = 1$ then 
    $$
    g \ker D_g^{\theta} = gf_0 K_{\chi z},
    $$
    where $ \chi$ is some inner function, $gf_0$ is a normalised element of $(g \ker D_g^{\theta} \ominus gC)$ and for $f \in \ker D_g^{\theta}$ we have $||gf||_{H^2}^2 = ||k ||_{H^2}^2$.
\end{enumerate}
\end{thm} Cancelling the $g$ and using the same notation as the previous theorem we obtain the following.
\begin{cor}
\begin{enumerate}
    \item If $\dim ( \ker D_g^{\theta} \ominus C) = 2$ then 
    $$
     \ker D_g^{\theta} = \begin{pmatrix} f_0 & h_0 \\ \end{pmatrix}
\begin{pmatrix}
k_0 \\
k_1 \\
\end{pmatrix}.$$
\item If  $\dim (\ker D_g^{\theta} \ominus C) = 1$ then 
    $$
    \ker D_g^{\theta} = f_0 K_{\chi z}.
    $$
\end{enumerate}
\end{cor}

 \section*{Declarations}
The author is grateful to the EPSRC for financial support. \newline The author is grateful to Professor Partington for his valuable comments.

\vskip 0.25cm
\noindent Conflicts of interest/Competing interests- not applicable.
\newline Availability of data and material- not applicable.
\newline Code availability- not applicable.
\newline Data sharing not applicable to this article as no datasets were generated or analysed during the current study.

\bibliographystyle{plain}
\bibliography{bibliography.bib}

\end{document}